\documentclass[preprint,authoryear]{elsarticle}
\usepackage{amsmath,amssymb,amsthm,enumerate,natbib,color,bbm,ifthen,graphicx,overpic}
\usepackage[draft]{fixme}

\def\Nset{{\mathbb{N}}}
\def\Rset{\mathbb R}

\def\C0{\mathsf{C}_0}

\newcommand{\as}{\text{a.s.} }
\newcommand{\iid}{i.i.d.}
\newcommand{\eqdef}{\ensuremath{\stackrel{\mathrm{def}}{=}}}
\def\eqsp{\;}

\newcounter{rmnum}

\newcommand{\rate}[2][]
{\ifthenelse{\equal{#1}{}}{\ratesymbol(#2)}{\ratesymbol^{(#1)}(#2)}}

\def\Xset{\mathsf{X}} 
\def\Xsigma{\mathcal{X}} 
\def\Yset{\mathsf{Y}} 
\def\Ysigma{\mathcal{Y}} 

\def\L{\mathcal{L}}

\def\PE{\ensuremath{\mathbb E}}
\def\PP{\ensuremath{\mathbb P}}

\newcommand{\tvnorm}[1]{\ensuremath{\left\|#1\right\|_{\mathrm{TV}}}}
\newcommand{\fnorm}[2]{\ensuremath{\left|#1\right|_{#2}}}

\newcommand{\supnorm}[1]{| #1 |_{\infty}}

\newcommand{\aslim}{\ensuremath{\stackrel{\text{a.s.}}{\longrightarrow}}}

\newtheorem{theo}{Theorem}[section]
\newtheorem{lemma}[theo]{Lemma}
\newtheorem{coro}[theo]{Corollary}
\newtheorem{prop}[theo]{Proposition}

\theoremstyle{remark}
\newtheorem{rem}{Remark}

\newcounter{hypA}
\newenvironment{hypA}{\refstepcounter{hypA}\begin{itemize}
  \item[{\bf A\arabic{hypA}}]}{\end{itemize}}

\newcounter{hypEE}

\newcounter{hypAM}

\newcounter{hypU}

\def\rmd{\mathrm{d}}

\def\1{\mathbbm{1}}

\newcommand{\CPE}[3][]
{\ifthenelse{\equal{#1}{}}{\mathbb{E}\left[\left. #2 \, \right| #3 \right]}{\mathbb{E}_{#1}\left[\left. #2 \, \right | #3 \right]}}
\newcommand{\CPP}[3][]
{\ifthenelse{\equal{#1}{}}{\mathbb{P}\left[\left. #2 \, \right| #3 \right]}{\mathbb{P}_{#1}\left(\left. #2 \, \right | #3 \right)}}

\newcommand{\wrt}{with respect to}

\newcommand{\coint}[1]{\left[#1\right)}

\newcommand{\ooint}[1]{\left(#1\right)}

\newcounter{hypoconbis}
\newcounter{saveconbis}
\newcommand\debutA{\begin{list} {\textbf{A\arabic{hypoconbis}}}{\usecounter{hypo
conbis}}\setcounter{hypoconbis}{\value{saveconbis}}}
\newcommand\finA{\end{list}\setcounter{saveconbis}{\value{hypoconbis}}}
\newcommand{\chunk}[4][]%
{\ifthenelse{\equal{#1}{}}{\ensuremath{{#2}_{#3:#4}}}{\ensuremath{#2}_{#1,#3:#4}}
}
\newcommand{\Var}[2][]{\ifthenelse{\equal{#1}{}}{\mathrm{Var} \left[ #2 \right]}{\mathrm{Var}_{#1} \left[ #2 \right]}}

\newcommand{\lpnorm}[3]{\ensuremath{\left\| #1 \right\|_{#2,#3}}} 
\newcommand{\sequencen}[1]{\left( #1 \right)_{n \in \Nset}}

\begin{document}

\begin{frontmatter}
\title{A simple variance inequality for U-statistics of a Markov chain with applications \tnoteref{t1}}
\journal{Statistics and Probability Letters}
\author[tpt]{G.~Fort \corref{cor1}}
\ead{gersende.fort@telecom-paristech.fr}
\author[tpt]{E.~Moulines}
\ead{eric.moulines@telecom-paristech.fr}
\author[pmc]{P.~Priouret}
\ead{p.priouret@upmc.jussieu.fr}
\author[mlv]{P.~Vandekerkhove}
\ead{Pierre.Vandekerkhove@univ-mlv.fr}

\cortext[cor1]{Corresponding Author}
\address[tpt]{LTCI-CNRS/TELECOM ParisTech, 46 rue Barrault, 75634 Paris Cedex 13, France. Fax: +(33) 1 45 81 71 44}
\address[pmc]{LPMA, Universit\'e Pierre et Marie Curie, Bo\^ite Courrier 188, 75252 Paris Cedex 5, France}
\address[mlv]{LAMA, Universit\'e de Marne-la-Vall\'ee 5, boulevard Descartes, 77454 Marne-la-Vall\'ee cedex 2, France}

\begin{abstract}
  We establish a simple variance inequality for U-statistics whose underlying
  sequence of random variables is an ergodic Markov Chain. The constants in
  this inequality are explicit and depend on computable bounds on the mixing
  rate of the Markov Chain. We apply this result to derive the strong law of
  large number for U-statistics of a Markov Chain under conditions which are
  close from being optimal.
\end{abstract}

\tnotetext[t1]{This work is partially supported by the French National Research Agency, under
    the program ANR-08-BLAN-0218 BigMC}

\end{frontmatter}

\section{Introduction}

Let $\{Y_{n}\}_{n=0}^{\infty}$ be a sequence of random variables with values in
a measurable space $(\Yset,\Ysigma)$.  Let $m$ be an integer and $h: \Yset^{m}
\to \Rset$ be a symmetric function. For $n \geq m$, the U-statistic associated
to $h$ is defined by
\begin{equation}
\label{eq:definition-U-stat}
U_{n,m}(h) \eqdef \binom{n}{m}^{-1} \sum_{1\leq i_{1} < \cdots < i_{m}\leq n}h(Y_{i_{1}},\ \dots,\ Y_{i_{m}}) \,.
\end{equation}
The function $h$ is often referred to as the kernel of the U-statistics and
$m$ is called the degree of $h$.  We refer to \cite{serfling:1980},
\cite{lee:1990}, and \cite{koroljuk:borovskich:1994} for U-statistics whose
underlying sequence is an \iid\ sequence of random variables.

Several authors have studied U-statistics for \emph{stationary} sequences of
dependent random variables under different dependence conditions: see
\cite{arcones:1998}, \cite{borovkova:burton:dehling:2001}, \cite{dehling:2006}
and the references therein. Much less efforts have been spent on the behavior
of U-statistics for non-stationary and asymptotically stationary processes; see
\cite{harel:puri:1990} and \cite{harel:elharfaoui:2008}. In this letter, we
establish a variance inequality for U-statistics whose underlying sequence is
an ergodic Markov Chain (which is not assumed to be stationary). This
inequality is valid for U-statistics of any order and the constants appearing
in the bound can be explicitly computed (for example, using Foster-Lyapunov
drift and minorization conditions if the chain is geometrically ergodic). This
inequality can be used to derive, with minimal effort, limit theorems for
U-statistics of a non-stationary Markov chain. In this paper, for the purpose
of illustration, we derive the strong law of large numbers (SLLN) under weak
conditions.

\section*{Notations}
Let $(\Yset, \Ysigma)$ be a general state space (see e.g. \cite[Chapter
3]{meyn:tweedie:2009}) and $P$ be a Markov transition kernel. $P$ acts on
bounded measurable functions $f$ on $\Yset$ and on measures $\mu$
on $\Ysigma$ via
\[
P f(x) \eqdef \int P(x,\rmd y) f(y) \eqsp, \qquad \mu P (A) \eqdef \int
\mu(\rmd x) P(x,A) \eqsp.
\]
We will denote by $P^n$ the $n$-iterated transition kernel defined by induction
\[
P^n(x,A) \eqdef \int P^{n-1}(x,\rmd y) P(y,A) = \int P(x,\rmd y) P^{n-1}(y,A)
\eqsp;
\]
where $P^0$ coincides with  the identity kernel.  For a function $V : \Yset \to \coint{1,+\infty}$, define the $V$-norm of a function $f: \Yset \to \Rset$ by
\[
\fnorm{f}{V} \eqdef \sup_{\Yset}|f|/V \eqsp.
\]
When $V=1$, the $V$-norm is the supremum norm and will be denoted by
$\supnorm{f}$.  Let $\L_V$ be the set of measurable functions such that $\fnorm{f}{V} <
+\infty $. For two probability measures $\mu_1,\mu_2$ on $(\Yset,\Ysigma)$,  $\tvnorm{\mu_1 -\mu_2}$ denotes
the total variation distance.

For $\mu$ a probability distribution on $(\Yset,\Ysigma)$ and $P$ a Markov
transition kernel on $(\Yset,\Ysigma)$, denote by $\PP_\mu$ the distribution of
the Markov chain $\sequencen{Y_n}$ with initial distribution $\mu$ and
transition kernel $P$; let $\PE_{\mu}$ be the associated expectation. For $p>0$
and $Z$ a random variable measurable \wrt\ the $\sigma$-algebra $\sigma\left(
  \sequencen{Y_n} \right)$, set $\lpnorm{Z}{\mu}{p}\eqdef \left(\PE_{\mu}\left[
    |Z|^p \right]\right)^{1/p}$.

\section{Main Results}
Let $P$ be a Markov transition kernel on $(\Yset,\Ysigma)$.  We assume that the
transition kernel $P$ satisfies the following assumption:
\begin{hypA}
\label{hyp:assumption-kernel}
The kernel $P$ is positive Harris recurrent and has a unique stationary
distribution $\pi$. In addition, there exist a measurable function $V: \Yset
\to \coint{1,+\infty}$ and a nonnegative non-increasing sequence $\left( \rho(k)
\right)_{k \in \Nset}$ such that $\lim_n \rho(n) = 0$ and for any probability
distributions $\mu$ and $\mu'$ on $(\Yset,\Ysigma)$, and any integer $k$,
\begin{equation}
\label{eq:contraction-V-norm}
\tvnorm{\mu P^k - \mu' P^k} \leq \rho(k) \left[ \mu(V) + \mu'(V) \right] \eqsp,
\end{equation}
and
\begin{equation}
\label{eq:finite-moment-pi}
\pi(V) < \infty \eqsp.
\end{equation}
\end{hypA}

\begin{hypA}
\label{hyp:bound-V-norm-canonical-U-stat}
The function $h$ is symmetric and $\pi$-canonical, {\it i.e,} for all
$(y_1,\dots,y_{m-1})\in \Yset^{m-1}$, $y \mapsto h(y_1,\dots,y_{m-1},y)$ is $\pi$-integrable and
\begin{equation}
\label{eq:degenere}
\int \pi(\rmd y) h(y_1,\dots,y_{m-1}, y)=0 \, .
\end{equation}
\end{hypA}

For $\mu$ a probability measure on $(\Yset,\Ysigma)$, we denote
\begin{equation}
\label{eq:definition-M}
M(\mu,V) \eqdef \sup_{k \geq 0} \mu P^k (V) \eqsp.
\end{equation}
Note that, under \textbf{A~\ref{hyp:assumption-kernel}}, for any probability
measure $\mu$ on $(\Xset,\Xsigma)$, $\pi(V) \leq M(\mu,V)$.  We can now state
the main result of this paper, which is an explicit bound for the variance of
bounded $\pi$-canonical U-statistics. The proof of
Theorem~\ref{theo:variance-U-stat} is given in
Section~\ref{sec:proof:variance:Ustat}.
\begin{theo}
\label{theo:variance-U-stat}
Assume
\textbf{A\ref{hyp:assumption-kernel}}-\textbf{A\ref{hyp:bound-V-norm-canonical-U-stat}}.
If $|h|_\infty<\infty $ then, for any initial probability measure $\mu$ on
$(\Yset,\Ysigma)$,
\begin{equation}
\label{eq:bounded-function}
\lpnorm{U_{n,m}(h)}{\mu}{2} \leq  C_{n,m} \sqrt{M(\mu,V)} \supnorm{h} \, n^{-m/2}
\end{equation}
with
\begin{equation}
\label{eq:definition-constant-C_m}
C_{n,m} \eqdef  \; 2^{m/2+1} \sqrt{(2m)!}\;        \left(\sum_{k=0}^{n} (k+1)^m \rho(k) \right)^{1/2}\, \frac{n^m }{\binom{n}{m}} \eqsp.
\end{equation}
\end{theo}
\begin{rem}
  In the case where $\rho(k)= \varrho^k$ for some $\varrho \in \ooint{0,1}$,
  for all $(m,n) \in \Nset$,
$$
\sum_{k=0}^{n} (k+1)^m \rho(k) \leq \frac{1}{\varrho \, (-\ln(\varrho))^{m+1} }
\frac{m ^{m+1}-(-\ln(\varrho))^{m+1}}{m +\ln (\varrho)} \eqsp.
$$
\end{rem} 
We may extend Theorem~\ref{theo:variance-U-stat} to symmetric functions $h$
which are not canonical.  For any integer $p$ and any $\mu_1, \dots, \mu_p$,
$p$ (signed) finite measures on $(\Yset,\Ysigma)$, denote by $\mu_1 \otimes
\dots \otimes \mu_p \eqdef \bigotimes_{i=1}^p \mu_i$, the product measure on
$(\Yset^p, \Ysigma^{\otimes p})$.  For $\mu$ a (signed) finite measure on
$(\Yset,\Ysigma)$, define $\mu^{\otimes p} \eqdef \mu \otimes \dots \otimes
\mu$.

Let $h:\Yset^m \to \Rset$ be a measurable and symmetric function such that
$\pi^{\otimes m} (|h|) < \infty$.  Define for any $c \in \{1,\dots,m-1\}$ the
measurable function $\pi_{c,m} h: \Yset^c \to \Rset$ given by
 \begin{equation}
 \label{eq:definition-pi-k-m}
 \pi_{c,m}h(y_{1}, \dots, y_{c}) \eqdef (\delta_{y_{1}}-\pi)\otimes \dots \otimes (\delta_{y_{c}}-\pi) \otimes \pi^{\otimes (m-c)}[h] \eqsp,
 \end{equation}
 where for any $y \in \Yset$, $\delta_y$ denotes the Dirac mass at $y$. Set
 \begin{equation}
 \label{eq:definition-pi-m-m}
\pi_{0,m} h \eqdef \pi^{\otimes m} h \quad \text{and} \quad  \pi_{m,m}h(y_{1}, \dots, y_{m}) \eqdef \bigotimes_{i=1}^m (\delta_{y_{i}}-\pi)[h]
 \eqsp.
 \end{equation}
 Note that for any $c \in \{1,\dots,m\}$ $\pi_{c,m}h$ is a $\pi$-canonical
 function.  The Hoeffding decomposition allows to write any U-statistics associated to
a symmetric function $h$ as the following sum of canonical U-statistics (see
 e.g.  \citep[p.  178, Lemma A]{serfling:1980}):
\begin{equation}
\label{eq:hoeffding-decomposition}
U_{n,m}(h) =   \sum_{c=0}^{m} \binom{m}{c}  U_{n,c}(\pi_{c,m}h) \eqsp,
\end{equation}
where $U_{n,c}$ is defined in \eqref{eq:definition-U-stat} when $c \geq 1$ and
$U_{n,0}(f)\eqdef f$. The symmetric function $h$ is said to be $d$-degenerated
(for $d \in \{0, \dots, m\}$) if $\pi_{d,m} h \not \equiv 0$ and $\pi_{c,m} h
\equiv 0$ for $c \in \{0, \dots, d-1 \}$. By construction, a $\pi$-canonical
function $h$ is $m$-degenerated (it is also said ``completely degenerated'').
\begin{coro}
\label{coro:variance-U-stat-non-canonique} Assume \textbf{A\ref{hyp:assumption-kernel}}.
Let $h$ be a bounded symmetric $d(h)$-degenerated function. Then
\[
\lpnorm{U_{n,m}(h) - \pi^{\otimes m}h}{\mu}{2} \leq \sqrt{M(\mu,V)} \supnorm{h}
\sum_{c=d(h)\vee 1}^m \binom{m}{c} 2^c \, C_{n,c} n^{-c/2} \eqsp,
\]
where $C_{n,c}$ is defined in \eqref{eq:definition-constant-C_m}.
\end{coro} It is possible to extend the previous result to unbounded canonical
functions.  Define, for any $q \geq 1$,
\begin{equation}
\label{eq:bound-h}
 B_q(h) \eqdef \sup_{(y_1,\dots,y_m) \in \Yset^m} \frac{|h(y_1, \cdots, y_m)|}{\sum_{j=1}^m V^{1/q}(y_j)}  \eqsp,
\end{equation}
where $V$ is defined in \textbf{A\ref{hyp:assumption-kernel}}.  The proof of
Corollary~\ref{coro:variance-U-stat} is given in
Section~\ref{sec:proof:variance:Ustat}.
\begin{coro}
\label{coro:variance-U-stat}
Assume \textbf{A\ref{hyp:assumption-kernel}}-\textbf{A\ref{hyp:bound-V-norm-canonical-U-stat}}
and that, for some $p \in \coint{0,\infty}$, $B_{2(p+1)}(h) < \infty$ holds.
Then, for any initial probability measure $\mu$ on $(\Yset,\Ysigma)$,
\begin{equation*}
\lpnorm{U_{n,m}(h)}{\mu}{2} \leq 2^{m/2} m \sqrt{(2m)!}  \; D(p,\mu,V,h) \;
\left( \sum_{k=0}^n (k+1)^m \left(\rho(k)\right)^{\frac{p}{(p+1)}}
\right)^{1/2} \frac{n^{m/2}}{\binom{n}{m}} \eqsp,
\end{equation*}
where the constant $D(p,\mu,V,h)$ is given by
\begin{multline}
\label{eq:definition-Dp}
D(p,\mu,V,h) \eqdef 2^{\frac{2p+1}{2(p+1)}} \; \left[ p^{\frac{1}{p+1}} +
  p^{-\frac{p}{p+1}} \right]^{1/2} \; \sqrt{M(\mu,V)} \; B_{2(p+1)}(h) \eqsp.
\end{multline}
\end{coro}
Using again the Hoeffding
  decomposition~(\ref{eq:hoeffding-decomposition}),
  Corollary~\ref{coro:variance-U-stat} can be extended to the case when $h$ is
  $d$-degenerated for $d \in \{0, \cdots, m-1\}$. Details are left to the
  reader. When used in combination with explicit ergodicity bounds
for Markov chains, Theorem~\ref{theo:variance-U-stat} and the corollaries can
be used to obtain non-asymptotic computable bounds for the variance of U- and
V-statistics. As a simple illustration, assume that the transition kernel $P$
is phi-irreducible, aperiodic and that
\begin{enumerate}
\item \emph{(Drift condition)} there exist a drift function $V: \Yset \to
  \coint{1, +\infty}$ and constants $1<b< \infty$, and $\lambda \in
  \ooint{0,1}$ such that
\begin{equation*}
P V \leq  \lambda V + b \eqsp.
\end{equation*}
\item \emph{(Minorization condition)} for any $d \geq 1$, the level sets $\{ V
  \leq d \}$ are petite for $P$.
\end{enumerate}
Then, there exists a probability distribution $\pi$ such that $\pi P = \pi$ and
$\pi(V) \leq b (1-\lambda)^{-1}$. In addition, there exist computable constants
$C < \infty$ and $\rho \in (0,1 )$ such that for any probability measures
$\mu,\mu'$ on $(\Yset,\Ysigma)$ and any $n \geq 0$,
$$
\tvnorm{\mu P^n - \mu' P^n} \leq C \ \rho^n \left[ \mu(V) + \mu'(V) \right] \eqsp;
$$
(see for example \cite{roberts:rosenthal:2004},
\cite{douc:moulines:rosenthal:2004} or \cite{baxendale:2005}).  Assumption
\textbf{A\ref{hyp:assumption-kernel}} is thus satisfied with $\rho(k)= C
\rho^k$ and we may thus apply Theorem~\ref{theo:variance-U-stat} to obtain a
non-asymptotic bound.

It can also be used to derive limiting theorems for U-statistics of Markov
chains. In what follows, as an illustration of our result, we derive a law of
large numbers which holds true under conditions which are, to the best of our
knowledge, the weakest known so far and more likely pretty close from being
optimal.
\begin{theo}
\label{theo:LGN-Ustat}
Assume \textbf{A\ref{hyp:assumption-kernel}} with
$\rho(n)=O\left(n^{-r}\right)$ for some $r > 1$.  Let $m \geq 1$ and $h:
\Yset^{m}\rightarrow \Rset$ be a symmetric function such that for some
$\delta>0$,
\begin{equation}
  \label{eq:Moment:LGN:Ustat}
  \sup_{(y_1,\dots,y_{m}) \in \Yset^m}
\frac{|h(y_1,\dots,y_m)|(\log^{+}|h(y_1,\dots,y_m)|)^{1+\delta}}{\sum_{i=1}^m
  V(y_i)} <\infty \eqsp.
\end{equation}
Then, for any probability measure $\mu$ on $(\Yset,\Ysigma)$ such that
$M(\mu,V) < \infty$,
\begin{equation}
  \label{eq:LGN-Limite1}
  \binom{n}{m}^{-1} \sum_{1\leq i_{1}<\cdots<i_{m}\leq n} \left\{ h(Y_{i_{1}},\dots,\
  Y_{i_{m}}) - \PE_\mu[h(Y_{i_{1}},\dots,\ Y_{i_{m}})] \right\} \to 0 \eqsp,
\quad \PP_\mu-\as\ \,
\end{equation}
when $n \to +\infty$ and
\begin{multline}
  \label{eq:LGN-Limite2}
\lim_{n \to +\infty} \binom{n}{m}^{-1} \sum_{1\leq i_{1}<\cdots<i_{m}\leq n}
\PE_\mu[h(Y_{i_{1}},\dots,\ Y_{i_{m}})]\\
= \int \pi(\rmd y_1) \cdots \pi(\rmd y_m) h(y_1, \cdots, y_m) \eqsp.
\end{multline}
\end{theo}

\section{Proof of Theorem~\ref{theo:variance-U-stat} and Corollary~\ref{coro:variance-U-stat}}
\label{sec:proof:variance:Ustat}
For any probability measure $\mu$ on $(\Yset,\Ysigma)$, for any positive
integer $\ell$ and any ordered $\ell$-uplet $k_0= 0 \leq k_1 \leq \dots \leq
k_\ell$, consider the probability measure $\PP_{\mu}^{k_1,k_2,\dots,k_\ell}$
defined for any nonnegative measurable function $f: \Yset^\ell \to \Rset_+$, by
\begin{equation}
\label{eq:probability-on-product}
\PP_{\mu}^{k_1,k_2,\dots,k_\ell}(f) \eqdef \idotsint \mu(\rmd y_0) \prod_{i=1}^\ell P^{k_i-k_{i-1}}(y_{i-1},\rmd y_{i}) f(\chunk{y}{1}{\ell}) \, ,
\end{equation}
where $\chunk{y}{1}{\ell}\eqdef (y_{1},\dots,y_\ell)$. Note that, by construction,
\[
\PE_\mu\left[ f(Y_{k_1},\dots,Y_{k_\ell}) \right]=  \PP_{\mu}^{k_1,k_2,\dots,k_\ell}(f) \eqsp.
\]
For any positive integer $m$ and any ordered $2m$-uplet $\mathcal{I}=(1 \leq i_1 \leq i_2 \leq \dots \leq i_{2m})$, we denote
for $\ell \in \{1,\dots,m\}$,
\begin{align}
\label{eq:definition-j}
&j_{\ell}(\mathcal{I}) \eqdef \min(i_{2\ell-1}-i_{2\ell-2}, i_{2\ell}-i_{2\ell-1}) \, ,\\
&\label{eq:definition-j-star} j_\star(\mathcal{I}) \eqdef
\max\left[j_1(\mathcal{I}),j_2(\mathcal{I}),\dots,j_m(\mathcal{I})\right]\eqsp,
\end{align}
where, by convention, we set $i_0=1$.
Denote by $\mathcal{B}_+(\Yset^{2m})$ the set of nonnegative measurable function $f: \Yset^{2m} \to \Rset_+$.
For any probability measure $\mu$ on $(\Yset,\Ysigma)$ and any ordered $2m$-uplet $\mathcal{I}$, denote
$\PP^{\mathcal{I}}_{\mu} \eqdef P^{i_1,\dots,i_{2m}}_\mu$. We consider
the probability measure $\tilde{\PP}_\mu^{\mathcal{I}}= \tilde{\PP}_\mu^{i_1,\dots,i_{2m}}$ on $(\Yset^{2m},\Ysigma^{\otimes 2m})$ given for $f \in \mathcal{B}_+(\Yset^{2m})$ by
\begin{equation}
\label{eq:definition-tilde-PP-1}
\tilde{\PP}^{\mathcal{I}}_\mu(f)
\eqdef \int \pi (\rmd y_{1}) \PP^{i_{2},\dots,i_{2m}}_{\mu} (\rmd \chunk{y}{2}{2m})f( \chunk{y}{1}{2m}) \eqsp,
\end{equation}
if $\inf \left\{ k \in \{1,\dots, m \}, j_\star(\mathcal{I}) = j_k(\mathcal{I}) \right\}=1$ and
\begin{equation}
\label{eq:definition-tilde-PP}
\tilde{\PP}^{\mathcal{I}}_\mu(f)
\eqdef \int \PP^{i_1,\dots,i_{2{\ell}-2}}_{\mu}(\rmd \chunk{y}{1}{2{\ell}-2})
\pi (\rmd y_{2{\ell}-1}) \PP^{i_{2{\ell}},\dots,i_{2m}}_{\mu}
(\rmd \chunk{y}{2{\ell}}{2m})f( \chunk{y}{1}{2m}) \eqsp,
\end{equation}
if $\ell=\inf \left\{ k \in \{1,\dots, m \}, j_\star(\mathcal{I}) = j_k(\mathcal{I}) \right\} \in \{2,\dots,m \}$.
For any permutation $\sigma$ on $\{1,\dots,2m\}$, define $f_\sigma: \Yset^{2m} \to \Rset$ the function
\begin{equation}
\label{eq:definition-f-sigma}
f_\sigma(y_1,\dots,y_{2m}) \eqdef  h\left(y_{\sigma(1)},\dots,y_{\sigma(m)}\right) h\left(y_{\sigma(m+1)},\dots,y_{\sigma(2m)}\right) \,.
\end{equation}
Since the function $h$ is $\pi$-canonical, it follows from the definition of $\tilde{\PP}_\mu^{\mathcal{I}}$ that, for
any ordered $2m$-uplet $\mathcal{I}$ and any permutation $\sigma$,
\begin{equation}
\label{eq:expectation-pi-canonical}
\tilde{\PP}^{\mathcal{I}}_\mu(f_\sigma) = 0 \eqsp.
\end{equation}
This relation plays a key role in all what follows and is the main motivation for considering the probability measures
$\tilde{\PP}^{\mathcal{I}}_\mu$.

\begin{prop}
\label{prop:bound-total-variation}
Assume
\textbf{A\ref{hyp:assumption-kernel}}-\textbf{A\ref{hyp:bound-V-norm-canonical-U-stat}}.
Then, for any probability measure $\mu$, any positive integer $n$ and any
ordered $2m$-uplet $\mathcal{I}$ in $\{1, \cdots, n\}$,
\begin{equation}
\label{eq:key-inequality-canonical-kernel}
\tvnorm{\PP_\mu^{\mathcal{I}}- \tilde{\PP}^{\mathcal{I}}_\mu} \leq 4 \, \rho\left(j_\star(\mathcal{I})\right) \, M(\mu,V)  \eqsp,
\end{equation}
where the sequence $(\rho(n))_{n \in \Nset}$, $M(\mu,V)$ and
$j_\star(\mathcal{I})$ are defined respectively in
\eqref{eq:contraction-V-norm}, \eqref{eq:definition-M}, and
\eqref{eq:definition-j-star}.
\end{prop}
\begin{proof}
  Let $ \mathcal{I}= (1 \leq i_1 \leq i_2 \leq \dots \leq i_{2m} \leq n)$. To
  simplify the notation, in what follows, the dependence in $\mathcal{I}$ of
  $j_1,\dots,j_m$ - defined in \eqref{eq:definition-j} - is implicit.  Assume
  first that $j_\star= j_1$.

Let $f \in \mathcal{B}_+(\Yset^{2m})$. The definition of \eqref{eq:probability-on-product} implies that
\[
\PP^{\mathcal{I}}_{\mu}(f) \eqdef P^{i_1,\dots,i_{2m}}_\mu(f) =\int \mu P^{i_1}(\rmd y_{1})
\PP^{i_2-i_1,\dots,i_{2m}-i_1}_{y_{1}}(\rmd \chunk{y}{2}{m})
 f(\chunk{y}{1}{2m}) \, .
\]
Combining this expression with the definition \eqref{eq:definition-tilde-PP-1}
of $\tilde{\PP}^{\mathcal{I}}_\mu$ yields
\begin{equation}
\label{eq:bound}
\left|\PP_\mu^{\mathcal{I}}(f)- \tilde{\PP}_{\mu}^{\mathcal{I}}(f) \right| \leq T_1 + T_2  \eqsp,
\end{equation}
with
\begin{align*}
  T_1 &\eqdef \left|\int \left[ \mu P^{i_1}(\rmd y_{1})-\pi(\rmd y_{1}) \right] \PP_\mu^{i_2,\dots,i_{2m}}(\rmd \chunk{y}{2}{2m}) f(\chunk{y}{1}{2m}) \right| \eqsp, \\
  T_2 & \eqdef \left|\int \mu P^{i_1}(\rmd y_{1}) \left[
      \PP^{i_2-i_1,\dots,i_{2m}-i_1}_{y_1}(\rmd \chunk{y}{2}{2m}) -
      \PP^{i_2,\dots,i_{2m}}_\mu(\rmd \chunk{y}{2}{2m})\right]
    f(\chunk{y}{1}{2m}) \right| \eqsp.
\end{align*}
Consider first $T_1$.  Since $\left| \int \PP_{\mu}^{i_2,\dots,i_{2m}} (\rmd
  \chunk{y}{2}{2m}) f(\chunk{y}{1}{2m}) \right| \leq \supnorm{f}$,
\textbf{A\ref{hyp:assumption-kernel}} and (\ref{eq:definition-M}) imply that
$$T_1\leq \tvnorm{\mu P^{i_1}-\pi} \supnorm{f} \leq \rho(i_1) \, \left[ \mu(V)
  + \pi(V) \right] \, \supnorm{f} \leq 2 \rho(i_1) M(\mu,V) \eqsp,
$$
where we have used that $\mu(V) \leq M(\mu,V)$ and $\pi(V) \leq M(\mu,V)$.
On the other hand, for any bounded measurable function $g: \Yset^{2m-1} \to
\Rset$, and $y \in \Yset$,
\[
\PP_{y}^{i_2-i_1,\dots,i_{2m}-i_1}(g) = \int \delta_y(\rmd y_1) P^{i_2-i_1}(y_1,\rmd y_2) \PP_{y_2}^{i_3-i_2,\dots,i_{2m}-i_2}(\rmd \chunk{y}{3}{2m}) g(\chunk{y}{2}{2m})
\]
and
\[
\PP_\mu^{i_2,\dots,i_{2m}} (g)= \int \mu P^{i_1}(\rmd y_1) P^{i_2-i_1}(y_1,\rmd y_2) \PP_{y_2}^{i_3-i_2,\dots,i_{2m}-i_2}(\rmd \chunk{y}{3}{2m}) g(\chunk{y}{2}{2m}) \eqsp.
\]
Therefore, under \textbf{A\ref{hyp:assumption-kernel}},
\[
\left| \PP_{y}^{i_{2}-i_{1},\dots,i_{2m}-i_{1}} (g) - \PP_\mu^{i_{2},\dots,i_{2m}} (g) \right|
\leq  \rho(i_{2}-i_{1}) \left[ V(y) + \mu P^{i_1} (V) \right]\supnorm{g} \eqsp.
\]
Integrating this bound shows that $T_2 \leq 2 \rho(i_2-i_1) M(\mu,V) \supnorm{f}$, where $M(\mu,V)$ is defined in \eqref{eq:definition-M}.
In conclusion we get
\begin{equation}\label{eq:V-ergodicity}
\left|\PP_\mu^{\mathcal{I}}(f)- \tilde{\PP}_{\mu}^{\mathcal{I}}(f) \right| \\
\leq 2 \left[\rho(i_2-i_1)+\rho(i_1) \right]\, M(\mu,V) \, \supnorm{f} \eqsp.
\end{equation}

Assume now that, for some $\ell \in \{2,\dots,m\}$, $j_\star= j_\ell$. With these notations, for any nonnegative function $f: \Yset^{2m} \to \Rset$,
\[
\PP^{\mathcal{I}}_\mu(f) =\int \PP^{i_1,\dots,i_{2\ell-1}}_{\mu}(\rmd \chunk{y}{1}{2\ell-1})   \PP^{i_{2\ell}-i_{2\ell-1},\dots,i_{2m}-i_{2\ell-1}}_{y_{2\ell-1}}
\left(\rmd \chunk{y}{2\ell}{2m}\right) f\left( \chunk{y}{1}{2m}\right) \eqsp.
\]
Combining this expression with the definition \eqref{eq:definition-tilde-PP} of
$\tilde{\PP}^{\mathcal{I}}_\mu$, we get
\begin{equation}
\left|\PP^{\mathcal{I}}_\mu(f)- \tilde{\PP}_{\mu}^{\mathcal{I}}(f) \right| \leq T_1 + T_2 \eqsp,\\
\end{equation}
with
\begin{multline*}
T_1 = \left|\int \PP^{i_1,\dots,i_{2\ell-2}}_{\mu}(\rmd \chunk{y}{1}{2\ell-2}) [P^{i_{2\ell-1}-i_{2\ell-2}}(y_{2\ell-2},\rmd y_{2\ell-1})-\pi(
\rmd y_{2\ell-1})] \right. \\
\left. \times \phantom{\int} \PP^{i_{2\ell},\dots,i_{2m}}_{\mu}(\rmd \chunk{y}{2\ell}{2m}) f(\chunk{y}{1}{2m})\right| \eqsp,
\end{multline*}
and
\begin{multline*}
T_2 = \left|\int \PP^{i_1,\dots,i_{2\ell-1}}_{\mu}(\rmd \chunk{y}{1}{{2\ell-1}}) \right.\\
\left. \phantom{\int} \times \left[ \PP^{i_{2\ell}-i_{2\ell-1},\dots,i_{2m}-i_{2\ell-1}}_{y_{2\ell-1}}(\rmd \chunk{y}{2\ell}{2m})-
\PP^{i_{2\ell},\dots,i_{2m}}_{\mu }(\rmd \chunk{y}{2\ell}{2m})\right]f(\chunk{y}{1}{2m})\right| \eqsp.
\end{multline*}
Consider first $T_1$.
Under \textbf{A\ref{hyp:assumption-kernel}}, \eqref{eq:contraction-V-norm}, for any $y_{2\ell-2} \in \Yset$, and any bounded
measurable function $g: \Yset \mapsto \Rset$,
\begin{multline*}
\int [P^{i_{2\ell-1}-i_{2\ell-2}}(y_{2\ell-2},\rmd y_{2\ell-1})-\pi( \rmd y_{2\ell-1})] g(y_{2\ell-1}) \\
\leq \rho\left(i_{2\ell-1}-i_{2\ell-2}\right) \, \left[ V(y_{2\ell-2}) + \pi(V) \right] \supnorm{g} \eqsp.
\end{multline*}
Applying this relation with
$$
g_{\chunk{y}{1}{2\ell-2}}(y_{2\ell-1})= \idotsint \PP^{i_{2\ell},\dots,i_{2m}}_{\mu}(\rmd \chunk{y}{2\ell}{2m}) f(\chunk{y}{1}{2\ell-2},y_{2\ell-1},\chunk{y}{2\ell}{2m}) \eqsp,
$$
and using that, for any $\chunk{y}{1}{2\ell-2} \in \Yset^{2\ell-2}$, $\supnorm{g_{\chunk{y}{1}{2\ell-2}}} \leq \supnorm{f}$, yields to
\begin{multline*}
T_1 \leq \rho(i_{2\ell-1}-i_{2\ell-2}) \left[ \mu P^{i_{2\ell-2}} (V) + \pi(V)\right] \supnorm{f} \\
\leq 2 \rho(i_{2\ell-1}-i_{2\ell-2}) M(\mu,V) \supnorm{f}  \eqsp.
\end{multline*}
Consider now $T_2$. Note that, for any bounded measurable function $g: \Yset^{2m-2\ell+1} \to \Rset$ that
\begin{multline*}
\PP_{y_{2\ell-1}}^{i_{2\ell}-i_{2\ell-1},\dots,i_{2m}-i_{2\ell-1}} (g) = \int   P^{i_{2\ell}-i_{2\ell-1}}(y_{2\ell-1},\rmd y_{2 \ell}) \\ \times \PP_{y_{2\ell}}^{i_{2\ell+1}-i_{2\ell},\dots,i_{2m}-i_{2\ell}}(\rmd \chunk{y}{2\ell+1}{2m}) g(\chunk{y}{2\ell}{2m}),\\
\end{multline*}
and
\begin{multline*}
\PP_\mu^{i_{2\ell},\dots,i_{2m}} (g) = \int \mu P^{i_{2\ell-1}} (\rmd y_{2\ell-1}) P^{i_{2\ell}-i_{2\ell-1}}(y_{2\ell-1},\rmd y_{2\ell}) \\ \times \PP_{y_{2\ell}}^{i_{2\ell+1}-i_{2\ell},\dots,i_{2m}-i_{2\ell}}(\rmd \chunk{y}{2\ell+1}{2m}) g(\chunk{y}{2\ell}{2m}) \eqsp.
\end{multline*}
Therefore, under \textbf{A\ref{hyp:assumption-kernel}},
for any bounded measurable function  $g: \Yset^{2m-2\ell+1} \to \Rset$ and $y_{2\ell-1} \in \Yset$,
\begin{multline*}
\left| \PP_{y_{2\ell-1}}^{i_{2\ell}-i_{2\ell-1},\dots,i_{2m}-i_{2\ell-1}} (g) - \PP_\mu^{i_{2\ell},\dots,i_{2m}} (g) \right|
\\
\leq \rho(i_{2\ell}-i_{2\ell-1}) \left[ V(y_{2\ell-1}) + M(\mu,V) \right]\supnorm{g} \eqsp.
\end{multline*}
Therefore, by integrating this bound \wrt\ $ \PP^{i_1,\dots,i_{2\ell-1}}_{\mu}$ yields to the bound
\[
T_2 \leq 2 \rho(i_{2\ell}-i_{2\ell-1}) M(\mu,V) \supnorm{f} \eqsp,
\]
which concludes the proof.
\end{proof}

\begin{lemma}
\label{lem:bound-Lp-norm}
Let $(\Xset,\Xsigma)$ be a measurable space.
Let $\xi$ and $\xi'$ be two probability measures on $(\Xset,\Xsigma)$ and $p \in \coint{0,+\infty}$.
Then, for any measurable function $f$ satisfying $\xi(|f|^{1+p}) + \xi'(|f|^{1+p}) < \infty$,
\[
\left| \xi(f) - \xi'(f) \right| \leq C(p) \left[ \xi(|f|^{1+p}) + \xi'(|f|^{1+p}) \right]^{1/(p+1)} \tvnorm{\xi - \xi'}^{p/(p+1)} \eqsp,
\]
where $C(p) \eqdef \left[ p^{1/(p+1)} + p^{-p/(p+1)} \right]$.
\end{lemma}
\begin{proof}
For any $M > 0$,
\begin{align*}
\left| \xi(f) - \xi'(f) \right|
&\leq  M \tvnorm{\xi-\xi'} \supnorm{f} + \xi\left[|f| \1 \{ |f| \geq M \}\right] + \xi'\left[ |f| \1 \{ |f| \geq M \} \right] \\
&\leq  M \tvnorm{\xi-\xi'} \supnorm{f} + M^{-p} \left[ \xi(|f|^{1+p}) + \xi'(|f|^{1+p}) \right] \eqsp.
\end{align*}
The proof follows by optimizing in $M$.
\end{proof}

\begin{prop}
\label{prop:bound-second-moment-canonical-U-stat}
Assume
\textbf{A\ref{hyp:assumption-kernel}-A\ref{hyp:bound-V-norm-canonical-U-stat}}.
Then, for any ordered $2m$-uplet $\mathcal{I}= (1 \leq i_1 \leq \dots \leq
i_{2m} \leq n)$, any permutation $\sigma$ on $\{1,\dots,2m\}$, and any initial
distribution $\mu$ on $(\Yset,\Ysigma)$,
\begin{equation}
\label{eq:bound-second-moment-canonical-U-stat-1}
\left| \PE_{\mu} \left[ f_\sigma(Y_{i_1}, \dots, Y_{i_{2m}}) \right] \right| \leq 4 M(\mu,V) \,  \rho\left(j_\star(\mathcal{I})\right) \, \supnorm{h}^2\eqsp,
\end{equation}
where the sequence $(\rho(n))_{n \in \Nset}$, the index $j_\star(\mathcal{I})$
and the function $f_\sigma$ are defined in \eqref{eq:contraction-V-norm},
\eqref{eq:definition-j-star}, and \eqref{eq:definition-f-sigma}, respectively.
If, for some $p \in \coint{0,\infty}$, the constant $B_{2(p+1)}(h)$, defined in
\eqref{eq:bound-h} is finite, then
\begin{equation}
\label{eq:bound-second-moment-canonical-U-stat-2}
\left| \PE_{\mu} \left[ f_\sigma(Y_{i_1}, \dots, Y_{i_{2m}}) \right]
\right| \leq  m^2 \  D(p,\mu,V,h)^2 \left(\rho\left(j_\star(\mathcal{I})\right)\right)^{\frac{p}{(p+1)}}
\end{equation}
where the constant $D(p,\mu,V,h)$ is defined in \eqref{eq:definition-Dp}.
\end{prop}

\begin{proof}
  The proof of \eqref{eq:bound-second-moment-canonical-U-stat-1} follows
  immediately from \eqref{eq:expectation-pi-canonical} and
  Proposition~\ref{prop:bound-total-variation}.

By applying the inequality $ab \leq 1/2(a^2+b^2)$ and the Jensen inequality, it follows from \textbf{A\ref{hyp:bound-V-norm-canonical-U-stat}}  that
\[
\left| f_\sigma(y_1,\dots,y_{2m}) \right|^{p+1} \leq (1/2) \, B^{2(p+1)}_{2(p+1)}(h) \, m^{2p+1} \,  \sum_{i=1}^{2m} V\left(y_{i}\right) \eqsp,
\]
where $f_\sigma$ is defined in \eqref{eq:definition-f-sigma}.  Therefore, for
any ordered $2m$-uplet $\mathcal{I}=(1 \leq i_1 \leq \dots \leq i_{2m} \leq
n)$,
\begin{align}
\label{eq:bound-Lp-norm-U-stat}
&\PP^{\mathcal{I}}_\mu\left[ \left| f_\sigma \right|^{p+1} \right] \leq   M(\mu,V)  B^{2(p+1)}_{2(p+1)}(h) \, m^{2(p+1)} \eqsp, \\
&\tilde{\PP}^{\mathcal{I}}_\mu\left[ \left| f_\sigma \right|^{p+1} \right]   \leq   M(\mu,V)  B^{2(p+1)}_{2(p+1)}(h) \, m^{2(p+1)} \eqsp.
\end{align}
The proof then follows by using \eqref{eq:expectation-pi-canonical} and by applying Proposition~\ref{prop:bound-total-variation} and Lemma~\ref{lem:bound-Lp-norm}.
\end{proof}

\begin{proof}[Proof of Theorem~\ref{theo:variance-U-stat} and Corollary~\ref{coro:variance-U-stat}]
Denote by  $\Gamma(2m)$ the collection of all permutations of $2m$ elements. We have
\begin{multline*}
  \PE_\mu\left[\left(\sum_{1\leq i_1 < \cdots < i_m\leq n} h(Y_{i_1}, \dots,
      Y_{i_m})\right)^2\right] \leq \\ \sum_{\sigma \in \Gamma(2m)}\sum_{1\leq
    i_1\leq\dots \leq i_{2m}\leq n} \left|\PE_\mu \left(h(Y_{i_{\sigma(1)}},
      \dots, Y_{i_{\sigma(m)}})h(Y_{i_{\sigma(m+1)}}, \dots,
      Y_{i_{\sigma(2m)}})\right)\right| \eqsp.
\end{multline*}
Let $k \geq 0$. Denote by $\mathsf{I}_{m,n}^k$ the set of all ordered
$2m$-uplet $\mathcal{I} = (1 \leq i_1 \leq \dots \leq i_{2m} \leq n)$ such that
$j_\star(\mathcal{I})=k$, where $j_\star(\mathcal{I})$ is defined in
\eqref{eq:definition-j-star}.  By definition, for $\mathcal{I} \in
\mathsf{I}_{m,n}^{k}$, and $\ell \in \{1,\dots,m\}$, $ j_\ell(\mathcal{I}) \leq
k$.  It is easily seen that the cardinal of $\mathsf{I}^{k}_{m,n}$ is at most
$2^m n^m (k+1)^m$.  The proof of Theorem~\ref{theo:variance-U-stat} follows
from Proposition~\ref{prop:bound-second-moment-canonical-U-stat},
\eqref{eq:bound-second-moment-canonical-U-stat-1}.

The proof of Corollary~\ref{coro:variance-U-stat} follows from
Proposition~\ref{prop:bound-second-moment-canonical-U-stat},
\eqref{eq:bound-second-moment-canonical-U-stat-2}.
\end{proof}
\section{Proof of Theorem~\ref{theo:LGN-Ustat}}
We will use the following elementary Lemma.
\begin{lemma}
\label{lem:easy-lemma}
Let $\sequencen{s_n}$ be a non-decreasing sequence of real numbers.
Let $\sequencen{u_n}$ be a non-decreasing sequence of positive numbers. Assume that
\begin{itemize}
\item the sequence $\sequencen{\ln(u_n)/\ln(n)}$ converges to a positive limit $\delta$.
\item for any $\alpha > 1$, the sequence $\sequencen{u^{-1}_{\lfloor \alpha^n \rfloor} s_{\lfloor \alpha^n \rfloor}}$ converges to $L$.
\end{itemize}
Then, the sequence $\sequencen{u_n^{-1} s_n}$ converges to $L$.
\end{lemma}
\begin{proof}
Let $\alpha > 1$. For any $n \in \Nset$, denote by $k_n \eqdef \sup \{ k \in \Nset, \lfloor \alpha^k \rfloor \leq n \}$.
Since the sequences $\sequencen{s_n}$  and $\sequencen{u_n}$ are non decreasing and $u_n > 0$ for any $n \in \Nset$,
\[
\frac{u_{\lfloor \alpha^{k_n} \rfloor}}{u_{\lfloor \alpha^{k_n+1} \rfloor}} \frac{s_{\lfloor \alpha^{k_n} \rfloor}}{u_{\lfloor \alpha^{k_n} \rfloor}} \leq \frac{s_n}{u_n} \leq \frac{u_{\lfloor \alpha^{k_n+1} \rfloor}}{u_{\lfloor \alpha^{k_n} \rfloor}} \frac{s_{\lfloor \alpha^{k_n+1} \rfloor}}{u_{\lfloor \alpha^{k_n+1} \rfloor}} \eqsp.
\]
Since $\lim_{n \to \infty} u_{\lfloor \alpha^{n+1} \rfloor}/ u_{\lfloor \alpha^{n} \rfloor}= \alpha^\delta$,
\[
\frac{1}{\alpha^\delta} L \leq \liminf_{n} \frac{s_n}{u_n} \leq \limsup_n \frac{s_n}{u_n} \leq \alpha^\delta L
\]
\end{proof}

\begin{proof}[Proof of Theorem~\ref{theo:LGN-Ustat}]
  Note that the positive and negative parts of $h$ satisfy the conditions of
  Theorem~\ref{theo:LGN-Ustat} so that we can assume without loss of generality
  that $h$ is non negative.

\begin{proof}[Proof of (\ref{eq:LGN-Limite2})]
  For any $\tau > 0$, denote
  \[ h_\tau(y_1,\dots,y_m) \eqdef h(y_1,\dots,y_m) \1_{\{|h(y_1,\dots,y_m)| \leq \tau\}} \eqsp.
  \]
   By A\ref{hyp:assumption-kernel}, we
  have, for any $1 \leq i_1 < \dots < i_m \leq n$,
  \[
  \left| \PE_{\mu}\left[ h_\tau(Y_{i_1},\dots, Y_{i_m}) \right] - \pi^{\otimes
      m}[h_\tau] \right| \leq 2 M(\mu,V) \supnorm{h_\tau} \sum_{j=1}^m
  \rho(i_j-i_{j-1}) \eqsp,
  \]
  where by convention, $i_0 =0$.  Note that $\sum_{1 \leq i_1 < i_2 \leq n}
  \rho(i_2-i_1) = \sum_{k=1}^{n-1} (n-k) \rho(k) \leq n \sum_{k=1}^{n-1}
  \rho(k)$.  Therefore,
\begin{align*}
& \left| \binom{n}{m}^{-1} \sum_{1 \leq i_1 < \dots < i_m \leq n}    \PE_{\mu}\left[ h_\tau(Y_{i_1},\dots, Y_{i_m}) \right] - \pi^{\otimes m}[h_\tau] \right|  \\
&  \leq 2 M(\mu,V) \tau \sum_{j=1}^m \binom{n}{m}^{-1} \sum_{1 \leq i_1 < \dots < i_m \leq n} \rho(i_j-i_{j-1}) \\
&  \leq 2 M(\mu,V) \tau \sum_{j=1}^m \binom{n}{m}^{-1} n^{m-2} \sum_{1 \leq i_{j-1} < i_j \leq n} \rho(i_j - i_{j-1}) \\
&  \leq 2 M(\mu,V) \tau \sum_{j=1}^m \binom{n}{m}^{-1} n^{m}  n^{-1} \sum_{k=1}^n \rho(k)  \eqsp,
\end{align*}
which goes to zero since $n^{-1} \sum_{k=1}^n \rho(k) \to 0$. Under the stated
assumptions, there exists a constant $C$ such that
\[
\PE_\mu\left[ |h(Y_{\chunk{i}{1}{m}})| \1_{ \{ |h(Y_{\chunk{i}{1}{m}})| \geq
    \tau \}}\right] \leq C \left( \log^+ \tau \right)^{-(1+\delta)} \eqsp.
\]
Since $\lim_{\tau \to \infty} \pi^{\otimes m}[h_\tau]= \pi^{\otimes m}[h]$, the proof follows.
\end{proof}
\begin{proof}[Proof of (\ref{eq:LGN-Limite1})] Let $m \geq 1$ be fixed. We prove
  that
  \begin{equation}
    \label{eq:LGN-Limite1-newcenter}
  \lim_n   \binom{n}{m}^{-1} \sum_{1\leq i_{1}<\cdots<i_{m}\leq n} h(Y_{\chunk{i}{1}{m}})
= \pi^{\otimes m} [h] \eqsp, \quad\ \PP_\mu-\as\ \,
  \end{equation}

  Using Lemma~\ref{lem:easy-lemma}, we have to prove that \eqref{eq:LGN-Limite1-newcenter} holds if for any
  $\alpha>1$,
  \begin{equation}
    \label{eq:LGN-Limite1-subsuite}
    \lim_{k \to +\infty}  \binom{\phi_k}{m}^{-1}   \sum_{1\leq i_{1}<\cdots<i_{m}\leq \phi_k} h(Y_{\chunk{i}{1}{m}})
 = \pi^{\otimes m}[ h]  \quad\ \PP_\mu-\as\ \,
  \end{equation}
where  $\phi_k \eqdef \lfloor \alpha^{k} \rfloor$. By the Hoeffding decomposition
\eqref{eq:hoeffding-decomposition}, it suffices to prove that for any $c \in
\{1, \cdots, m \}$,
\[
\lim_{k \to +\infty} \binom{\phi_k}{c}^{-1} \ \sum_{1\leq i_{1}<\dots<i_{c} \leq \phi_k}
\pi_{c,m} h(Y_{\chunk{i}{1}{c}}) \aslim 0
\]
where $\pi_{c,m} h$ is the symmetric $\pi$-canonical function defined in \eqref{eq:definition-pi-k-m};
note that under (\ref{eq:Moment:LGN:Ustat}),
  \begin{equation}
    \label{eq:moment-condition:gc}
    \sup_{(y_1, \cdots, y_c) \in \Yset^c} \frac{|\pi_{c,m}h(y_1,\cdots,y_c)| \
  \log^+(|\pi_{c,m}h(y_1,\cdots, y_c)|)^{1+\delta}}{\sum_{i=1}^c V(y_i)} < +\infty
\eqsp.
  \end{equation}
  The case $c=1$ is the ergodic theorem for Markov Chain (see for example
  \cite[Theorem 17.1.7]{meyn:tweedie:2009}).

We consider now the case $c \in \{2, \cdots, m\}$. 
In all what follows, the index $c \in \{2,\dots,m\}$ is given and for ease of notations, we denote by
$g$ an arbitrary $\pi$-canonical symmetric function of $c$ variables . Take $s>0$ such that
\begin{equation}
\label{eq:condition-tau-r}
2s <r-1 \,.
\end{equation}
By A\ref{hyp:assumption-kernel} and (\ref{eq:moment-condition:gc}), there
exists a constant $C$ depending upon $s$ and $M(\mu,V)$, such that
\begin{multline}
\label{eq:borel-cantelli}
\PE_\mu\left[ \sum_{k=1}^{\infty} \ \phi_k^{-c} \sum_{1\leq i_{1}<\dots < i_{c} \leq \phi_k} |g(Y_{\chunk{i}{1}{c}})|\1_{\{|g(Y_{\chunk{i}{1}{c}})|\geq \phi_k^{s}\}} \right] \\
\leq C \ \sum_{k=1}^{\infty}( \log \phi_k)^{-\delta-1} \eqsp,
\end{multline}
and the RHS is finite since $\alpha>1$ and $\delta >0$. Therefore,
\begin{equation}
\label{eq:almost-sure-truncated}
\phi_k^{-c} \sum_{1\leq i_{1}<\dots<  i_{c} \leq \phi_k} g(Y_{\chunk{i}{1}{c}})\1_{\{|g(Y_{\chunk{i}{1}{c}})|\geq \phi_k^{s} \}
}\rightarrow 0 \quad \PP_\mu-\as\ \, .
\end{equation}
We must now prove that
\begin{equation}
\label{eq:LGN-borne}
\lim_k \phi_k^{-c} \sum_{1\leq i_{1}<\dots< i_{c} \leq \phi_k} g_{\phi_k^s}(Y_{\chunk{i}{1}{c}}) = 0 \,, \quad  \PP_\mu-\as\ ,
\end{equation}
where for $\tau > 0$, $g_{\tau}(\chunk{y}{1}{c}) \eqdef g(\chunk{y}{1}{c})
\1_{\{ |g(\chunk{y}{1}{c})|<\tau\}}$.  We apply again the Hoeffding
decomposition~(\ref{eq:hoeffding-decomposition}) to the function $g_{\phi_k^s}$.
Observe that since $g$ is $\pi$-canonical,
satisfies~(\ref{eq:moment-condition:gc}) and $\pi(V) < +\infty$, the dominated
convergence theorem implies that $\lim_k \pi^{\otimes c}(g_{\phi_k^s}) =
\pi^{\otimes c}(g) =0$.  Hence, by (\ref{eq:hoeffding-decomposition}), the
limit (\ref{eq:almost-sure-truncated}) holds provided for any $\ell \in \{1,
\cdots,c \}$,
\begin{equation}
  \label{eq:sublimit-truncated-upper}
  \lim_{k \to \infty} \phi_k^{-\ell} \sum_{1 \leq i_1 <
  \cdots < i_\ell \leq \phi_k} \pi_{\ell,c}[g_{\phi_k^s}](Y_{\chunk{i}{1}{\ell}}) = 0 \,, \quad
\PP_\mu-\as\ .
\end{equation}
Since $g$ is $\pi$ canonical, for $\ell \in \{1, \cdots, c-1 \}$,  we have $\pi_{\ell,c} g=0$ which implies
\[
\pi_{\ell,c} [g_{\phi_k^s}] = \pi_{\ell,c}\left[g -g \1_{\{|g| \geq \phi_k^s\}}\right] = - \pi_{\ell,c}\left[g \1_{\{|g| \geq \phi_k^s\}}\right] \eqsp.
\]
Therefore, (\ref{eq:sublimit-truncated-upper}) is equivalent to
\[
\lim_{k \to \infty} \phi_k^{-\ell} \sum_{1 \leq i_1 <
  \cdots < i_\ell \leq \phi_k} \pi_{\ell,c}[g \1_{\{|g| \geq \phi_k^s\}}] = 0
\,, \quad \PP_\mu-\as\ \eqsp,
\]
which holds true by using an argument similar to \eqref{eq:borel-cantelli};
details are omitted.

When $\ell = c$, by definition of $\pi_{c,c}$ (see
(\ref{eq:definition-pi-m-m})) we have by applying
Theorem~\ref{theo:variance-U-stat}
\begin{multline*}
\PE_\mu\left[ \left(\phi_k^{-c} \sum_{1\leq i_{1}<\dots<i_{c} \leq \phi_k}
      \; \pi_{c,c}[g_{\phi_k^s}] (Y_{\chunk{i}{1}{c}}) \right)^{2}\right] \\ \leq C \ \phi_k^{-c} \left( \sum_{j=0}^{\phi_k} (j+1)^c \rho(j)
  \right) \phi_k^{2s} \leq C' \ \phi_k^{1-r+2s},
\end{multline*}
which by \eqref{eq:condition-tau-r} implies (\ref{eq:sublimit-truncated-upper})
when $\ell =c$. This concludes the proof.
\end{proof}
\end{proof}

\end{document}